\documentclass{amsart}

\usepackage{amsmath,amsfonts,amssymb}
\usepackage{amsthm}
\usepackage{ifthen}
\usepackage{longtable}
\usepackage{array}
\usepackage{url}
\usepackage{multirow}
\usepackage{tikz}
\numberwithin{equation}{section}

\newcommand{\MM}{{\mathcal M}}
\newcommand{\N}{{\mathbb N}}

\newcommand{\Q}{{\mathbb Q}}

\newcommand{\Z}{{\mathbb Z}}

\newcommand{\U}{{\mathcal U}}

\newcommand{\dv}{{\bold v}}

\newcommand{\D}{\text{$\mathcal{D}$}}

\newcommand{\F}{\text{$\mathbb{F}$}}
\newcommand{\G}{\text{$\mathbb{G}$}}
\newcommand{\C}{\text{$\mathcal{C}$}}
\renewcommand{\H}{\text{$\mathbb{H}$}}

\newcommand{\XX}{\mathbb X}
\newcommand{\YY}{\mathbb Y}
\newcommand{\E}{\mathcal E}
\newcommand{\na}{\nu^\ast}
\newcommand{\ma}{\mu^\ast}
\newcommand{\A}{\mathbb A}
\newcommand{\la}{\ell^\ast}

\newtheorem{theorem}{Theorem}[section]
\newtheorem{example}[theorem]{Example}
\newtheorem{definition}[theorem]{Definition}
\newtheorem{lemma}[theorem]{Lemma}
\newtheorem{corollary}[theorem]{Corollary}
\newtheorem{proposition}[theorem]{Proposition}

\theoremstyle{remark}
\newtheorem{remark}[theorem]{Remark}

\begin{document}

\title{Measure density for set decompositions and uniform distribution}

\dedicatory{Dedicated to W. G. Nowak on the occasion of his $60$th birthday}


\author[M.R. Iac\`o]{Maria Rita Iac\`o}
\address{M.R. Iac\`o \newline
\indent Graz University of Technology, \newline
\indent Institute of Mathematics A,\newline
\indent  Steyrergasse 30, 8010 Graz, Austria.}
\email{iaco\char'100math.tugraz.at}

\author[M. Pasteka]{Milan Pa\v{s}t\'eka}
\address{M. Pa\v{s}t\'eka \newline
\indent Pedagogick\'a fakulta TU v Trnava, \newline
\indent Priemyseln\'a 4, P.O. BOX 9, Sk-918 43, Trnava, Slovakia.}
\email{milan.pasteka\char'100truni.sk}

\author[R. Tichy]{Robert F. Tichy}
\address{R. F. Tichy \newline
\indent Graz University of Technology, \newline
\indent Institute of Mathematics A,\newline
\indent  Steyrergasse 30, 8010 Graz, Austria.}
\email{tichy\char'100tugraz.at}

\thanks{The first and third author are supported by the Austrian Science Fund (FWF): Project F5510, which is a part
of the Special Research Program "Quasi-Monte Carlo Methods: Theory and Applications"}

\begin{abstract}
The aim of this paper is to extend the concept of measure density introduced by Buck for finite unions of arithmetic progressions, to arbitrary subsets of $\N$ defined by a given system of decompositions. This leads to a variety of new examples and to applications to uniform distribution theory.
\end{abstract}

\maketitle

\section{Introduction and notation}

Let $S \subset\mathbb{N}$ be a subset of the set of positive integers. Then the limit

\begin{equation*}
d(S) = \lim_{N \to \infty}\frac{\#\{n \le N; n \in S\}}{N}
\end{equation*}

(if it exists) is called the {\it asymptotic density} of $S$.
Let us fix a positive integer $m \in \N$ and $a \in \N \cup \{0\}$. Clearly,

\begin{equation*}
a+(m)=\{x \in \N; x \equiv a\ ({\rm mod}\ m) \}\ ,
\end{equation*}
is an arithmetic progression, and $d(a+(m)) = \frac{1}{m}$.\\ Starting from the asymptotic density of finite unions
of arithmetic progressions, Buck \cite{Buck} defined the set function

\begin{equation*}
\ma(S)= \inf\left\{ \sum_{k=1}^n \frac{1}{m}_k;\ S \subset \bigcup_{k=1}^n a_k+(m_k) \right\}\ ,
\end{equation*}
now called {\it Buck measure density}. In general, $d(S) \le \ma(S)$ holds, but there are several examples of sets $S$ such that $d(S) \neq \ma(S)$.\\ The system of sets defined by

\begin{equation*}
\D_\mu = \{S \subset \N;\ \ma(S) + \ma(\N \setminus S) = 1 \}
\end{equation*}

is an algebra of sets and its elements are called {\it Buck measurable sets}. Moreover, the restriction $\mu = \ma|_{\D_\mu}$ is a finitely additive measure.\\ Our aim is to extend the definition of $\mu(S)$ to a bigger class of sets $S$; in this context $\mu$ extends to a $\sigma$-additive measure.

Let $\{s_n\}_{n\in\mathbb{N}}$ be a given sequence of numbers. Then the "counting set function" $A(S; \{s_n\})$ is defined to be the set of positive integers given by

\begin{equation}
A(S; \{s_n\}) =\{n \in \N; s_n \in S \}\ .
\end{equation}
A sequence of positive integers $\{s_n\}$ is called {\it uniformly distributed in $\Z$} (for short u.d. in $\Z$, details see in \cite{Niven}) if and only if for every arithmetic progression $a+(m)$ we have 

\begin{equation*}
d(A(a+(m)); \{s_n\}) = \frac{1}{m}\ .
\end{equation*}
The following characterization of Buck measurability is proved in \cite[Theorem 7, page 51]{Pasteka}.

\begin{theorem}\label{thm1}
 A set $S \subset \N$ belongs to $\D_\mu$ if and only if $d(\A(\{s_n\},S))=\ma(S)$ holds for every uniformly distributed sequence $\{s_n\}_{n\in\mathbb{N}}$ in $\Z$.
\end{theorem}
It is well-known that the uniform distribution property, introduced by H. Weyl \cite{Weyl} for sequences of real numbers in the unit interval, naturally extends to sequences on compact Hausdorff spaces and in topological groups (see e.g. \cite{DT, KN, SP}). In \cite{Tichy} the author provides a criterion for the uniform distribution of sequences in compact metric spaces. Let $(\MM, \rho, P)$  be a compact metric space, with metric $\rho$ and a Borel probability measure $P$. A sequence $\{x_n\}_{n\in\N}$ in $\MM$ is called
{\it Buck uniformly distributed} (for short B.u.d.) if and only if for every measurable set $H \subset \MM$ with $P(\partial H)=0$  we have $A(H, \{x_n\}) \in \D_\mu$ and $\mu(H, A(\{x_n\}))=P(H)$ (here $\partial H$ denotes as usual the boundary of $H$).

The following theorem, proved in \cite{Pasteka2}, is an analogue of Weyl's criterion for B.u.d. sequences $\{x_n\}_{n\in\mathbb{N}}\in \MM$. 

\begin{theorem}
A sequence $\{x_n\}_{n\in\mathbb{N}}$ in $\MM$ is Buck uniformly distributed if and only if
for every continuous real valued function $f$ defined on $\MM$ and for every sequence of positive integers $\{s_n\}_{n\in\N}$  uniformly distributed in $\Z$ we have

\begin{equation*}
\lim_{N \to \infty}\sum_{n=1}^N f(x_{s_n}) = \int_\MM f dP.
\end{equation*}
\end{theorem}

In the same paper, the author proves the existence of B.u.d. sequences in $\MM$.

\section{General results}
Let $\XX$ be an arbitrary set. For a fixed $n\in\mathbb{N}$, we denote by $\E_n =\{A_1^{(n)}, \dots,A_{k_n}^{(n)} \}$ a system of disjoint decompositions of $\XX$, i.e. $A_i^{(n)}\cap A_j^{(n)}=\emptyset$ and $\cup_{i=1}^{k_n}A_i^{(n)}=\XX$. Each system of decompositions satisfies the following conditions extending the properties of arithmetic progressions:

\begin{itemize}
\item [({\bf i})] For every family of sets $A_{h_1}^{(j_1)},\dots,A_{h_m}^{(j_m)}$ there exists an $s$ such that each of these sets is a union of sets belonging to $\E_s$.

\item [({\bf ii})] If $\{h_n\}$ is an arbitrary sequence of indices, then the intersection $\bigcap_{n=1}^\infty A_{h_n}^{(n)}$ contains at most $1$ element.
\end{itemize}

Let us denote by $\D_0$ the system of all sets of the form $A_{h_1}^{(j_1)}\cup...\cup A_{h_m}^{(j_m)}$. Condition ({\bf i}) assures that $\D_0$ is an algebra of sets and let $\Delta :\D_0 \to [0,1]$ be a finitely additive probability measure defined on $\D_0$. The set function

\begin{equation*}
\na(S) = \inf \{\Delta(A); A \in \D_0, S \subset A \}
\end{equation*}
will be called the {\it measure density} of the set $S$.
Let us remark that this is the standard way of constructing the outer measure $\nu^*$ starting from the finitely additive measure $\Delta$. In particular, the following result holds.

\begin{theorem}\label{thm1}
Let $\{c_n\}_{n\in\mathbb{N}}$ be a sequence in $\N$ such that for every $A \in \D_0$,
there exists $n_0$ such that $A$ is a union of sets from $\E_{c_n}$, for $n \ge n_0$. Then for arbitrary
$S \subset \XX$ we have

\begin{equation*}
\na(S) = \lim_{n \to \infty} \sum_{S \cap A_j^{(c_n)}\neq \emptyset} \Delta(A_j^{(c_n)})\ .
\end{equation*}
Moreover, if a set $S \subset \XX$ has non-empty intersection with every set from
$\E_n, n=1,2,\dots$, then $\na(S)=1$.
\end{theorem}
A set $S \subset \XX$ is called {\it $\na$-measurable} if and only if $\na(S)+ \na(\XX \setminus S) = 1$.
We denote the Carath\'eodory extension of $\D_0$ by $\D_\nu$. By definition it is the system of all $\na$-measurable subsets of $\XX$.\\
It follows from general measure theory that the system $\D_\nu$ is an algebra of sets and the restriction $\nu = \na|_{\D_\nu}$ is a finitely additive probability measure on $\D_\nu$ (see e.g. \cite{Halmos}).\\

Now we provide some examples of systems of decompositions and related systems $\D_\nu$ of $\na$-measurable subsets.

\begin{example}
Take $\XX=\N$ and $\E_n=\{A_1^{(n)},\dots, A_{n!}^{(n)} \}, n \in \N$,
where $A_j^{(n)}= j-1 +(n!), j=1,2,\dots,n!$. Then $\na$ is the Buck measure density defined in \cite{Buck}.
\end{example}
\begin{example}\label{example2}
Again let $\XX=\N$. Consider the system of decompositions
$\E_n=\{A_1^{(n)},\dots, A_{n}^{(n)} \}, n \in \N$.
where $A_1^{(n)}= \N \setminus \{1,\dots, n-1\}, A_j^{(n)}=\{j-1\}$, $j=2,\dots,n$.
In this case $\D_0$ consists of all subsets of $\N$ which are finite or have finite complement. Let $S \in \D_0$, put
$\Delta(S)=1$ if $S$ is infinite and $\Delta(S)=0$ for $S$ finite. In this case $\na(A)=1$ if and only if $A$ is an infinite set, and the system
$\D_\nu$ coincides with $\D_0$.
\end{example}

\begin{example}\label{example3}
Let $\XX = [0,1) \cap \Q$, and $\E_n=\{[\frac{k-1}{n},\frac{k}{n})\cap \Q,\ k=1,\dots,n \}, n \in \N$ and
$\Delta([\frac{k-1}{n},\frac{k}{n})\cap \Q)= \frac{1}{n}$. Then $\na$ is the Jordan upper measure defined on the system of subsets
of $\XX$. Here $\D_\nu$ is the system of Jordan measurable sets.
\end{example}

It is well-known that for every compact group there exists a probability measure defined on the system of its Borel subsets, invariant
with respect the group operation (the normalized Haar measure, see for instance \cite{Halmos, Hewitt_Ross}). In \cite{FPTW} the authors study certain finitely additive measures on topological groups and rings.
\begin{example}
Let us first assume that $\XX=\G$ is an infinite multiplicative locally compact abelian group and $\H$ a subgroup of finite index. Then the quotient $\G / \H$ is a finite group. Denote by $\Delta^\ast$ the normalized Haar measure measure on $\G$. Since it is invariant with respect the group operation we have

\begin{equation*}
\Delta^\ast(\H) = \frac{1}{|\G/\H|}\ ,
\end{equation*}
where  $|X|$ denotes the cardinality of $X$.\\
Let $S=\{\H_n; n=1,2,\dots\}$ be a system of subgroups of $\G$ of finite index such that $\H_i\cap \H_j=\H_k$, for every $i,j,k$ and $\cap_{n=1}^\infty H_n =\{e\}$, where $e$ is the neutral element of $\G$. Thus for every $n$ we have a finite decomposition

\begin{equation*}
\E_n =\{a_1^{(n)}H_n, \dots,a_{k_n}^{(n)}H_n \} , \  k_n =|\G / \H_n|.
\end{equation*}
The system $\D_0$ consists of all sets of the form $g_1\H_n\cup \dots\cup g_k\H_n$, $g_i \in \G$, $n=1,2,\dots$. Let $\Delta$ be the restriction of $\Delta^\ast$ to $\D_0$ and $\na$  the corresponding measure density.\\

Let $\G$ be the free abelian group with countable set of generators
$\{p_1, p_2, \dots\}.$ Let $\H_n$, $n=1,2,\dots$, be the subgroups generated by $\{p_1^n, p_2^n,\dots, p_n^n, p_{n+1},p_{n+2},\dots\}$. Since every element of $\G$ can be written as product of a finite number of generators, we get the following disjoint decomposition

\begin{equation*}
\G= \bigcup_{0\le j_i < n} p_1^{j_1}p_2^{j_2} \dots p_n^{j_n}\H_n.
\end{equation*}
Thus $\G / \H_n$ contains $n^n$ elements and therefore $\Delta(a\H_n) = \frac{1}{n^n}$, for $a \in \G$.\\

In particular, if $\G=\Q^\ast$ is the multiplicative group of positive rational numbers, then it can be considered as the free abelian group generated by all primes. In this case, the measurability is not compatible with the natural order relation on $\Q$. The inclusion

\begin{equation*}
(0,1]\cap \Q^+ \subset \cup_{i=1}^m a_i \H_n
\end{equation*}
implies that the numbers $a_i$ take all the values $p_1^{j_1}p_2^{j_2} \dots p_n^{j_n}, 0\le j_i <n$, thus $\na((0,1]\cap \Q^+)=1$.
Analogously, we can show that $\na((1, \infty)\cap \Q^+)=1$. Thus these sets are not measurable.
\end{example}

\begin{example}
If $\G =\prod_{j=1}^\infty \G_j$ is the direct product of finite groups, then we can take $\H_n=\prod_{j=n+1}^\infty \G_j$ and in this case

\begin{equation*}
\Delta(a\H_n)= \frac{1}{|G_1\cdot...\cdot G_n|}, \ a \in \G\ .
\end{equation*}
\end{example}

Let us return to the general setting. We will start by constructing a compact metric space containing $\XX$ as dense subset. Then we define
a Borel probability measure induced by $\Delta$.\\

First, we define the metric on $\XX$ based on the system of decompositions $\E_n, n=1,2,\dots$.
Let $x,y \in \XX$ and put $\psi_n(x,y)=0$ if $x,y$ belong to the same set of $\E_n$, and
$\psi_n(x,y)=1$ otherwise (for $n=1,2,\dots$). Define

\begin{equation*}
\rho(x,y)=\sum_{n=1}^\infty \frac{\psi_n(x,y)}{2^n}\ ,
\end{equation*}
and $\rho$ is a metric on $\XX$.
In particular, 
\begin{equation}\label{1}
\rho(x,y) \le \frac{1}{2^N}
\end{equation}
if and only if $x,y$ belong to the same set of every decomposition $\E_n, n=1,\dots,N$. From condition ({\bf i}) it follows that a sequence $\{x_n\}_{n\in\mathbb{N}}$ of elements in $\XX$ converges to an element
$x \in \XX$ if and only if for every $s=1,2,\dots$ there exists $n_0$ such that for every $n \ge n_0$ the elements $x_n$ and $x$ belong to the same set of $\E_s$.\\
Similarly, one can define the concept of Cauchy sequence which leads to the completion of $\XX$ in the usual way.
Let $\bar{\XX}$ be the completion of the metric space $(\XX, \rho)$ and for $S \subset \bar{\XX}$ let $\bar{S}$ be its closure in $\bar{\XX}$. Then, clearly

\begin{equation*}
\bar{\XX} = \bar{A}_1^{(n)}\cup \dots\cup \bar{A}_{k_n}^{(n)}
\end{equation*}
for $n=1,2,\dots$. Since a sequence of elements of $\XX$ is defined to be fundamental if and only if for every $s=1,2,\dots$ there exists $n_0$ such that for $m, n \ge n_0$ the elements $x_m$ and $x_n$ belong to the same set of $\E_s$, then the sets $\bar{A_1^{(n)}},\dots, \bar{A}_{k_n}^{(n)}$,$n=1,2,\dots$ are disjoint. Thus they are open and closed and, by condition ({\bf i}) and inequality \eqref{1}, it follows that for every $N=1,2,\dots$ there exists a finite
$\frac{1}{2^N}$-net. This shows that the metric space $\bar{\XX}$ is compact. 

We construct a $\sigma-$additive Borel probability measure on $\bar{\XX}$. The compactness of $\bar{X}$ implies that the extension of $\Delta$ to sets of the form $\{\bar{A}; A\in D_0\}$ which are open and closed is a $\sigma$-additive probability measure, since $\Delta(\bar{A})=\Delta(A)$. Then

\begin{equation*}
P^\ast(B)= \inf\left\{\sum_{n=1}^\infty \Delta(\bar{A}_n); B \subset \bigcup_{n=1}^\infty \bar{A}_n,\ \bar{A}_n \in {\mathcal S} \right\}
\end{equation*}
is an outer measure on $\bar{\XX}$ and ${\mathcal S}_{P^\ast}= \{B; P^\ast(B)+P^\ast(\bar{\XX} \setminus B)=1 \}$, the system of $P^\ast$-measurable sets, is a $\sigma$-algebra. Therefore, the restriction $P$ of $P^\ast$ on ${\mathcal S}_{P^\ast}$ is a
$\sigma$-additive probability measure on ${\mathcal S}_{P^\ast}$.
Moreover, since ${P^\ast}$ is, by definition, an outer measure, ${\mathcal S}_{P^\ast}$ contains all open sets. %In fact, if $B_1, B_2 \subset \bar{\XX}$ and for some positive constant $c$, $\rho(\alpha, \beta)< c$, for $\alpha \in B_1, \beta \in B_2$ holds, then
%\begin{equation*}
%P^\ast(B_1 \cup B_2) = P^\ast(B_1) + P^\ast(B_2)\ .
%\end{equation*}
Thus $P$ is a Borel probability measure on the compact metric space $\bar{\XX}$.\\

Following the usual procedure, we have a compact metric space and a Borel probability measure defined on it. We can introduce a suitable definition of uniform distribution of a sequence $\{\alpha_n\}_{n\in\mathbb{N}}$ in $\bar{\XX}$ with respect to $P$, namely Buck uniform distribution.\\
Since the set $\XX$ is dense in its completion, there exists a sequence $\{x_n\}_{n\in\mathbb{N}}$ in $\XX$ such that $\lim_{n\to \infty}\rho(x_n, \alpha_n)=0$. Since every continuous function on $\bar{\XX}$ is uniformly continuous, $\{x_n\}_{n\in\mathbb{N}}$ is also a B.u.d. sequence. Considering a set $C$ with $\bar C\in\mathcal S$ and $\partial\bar{C}=\emptyset$, yealds $A(C, \{x_n\}) \in \D_\mu$ and

\begin{equation*}
\mu(A(C,\{x_n\})) = \Delta(C).
\end{equation*}
A sequence of elements of $\XX$ fulfilling this condition will be called {\it $\na$-B.u.d.}. Moreover, it is easy to see that for every $S\in\D_\nu$ the set $A(S, \{x_n\})$ is measurable in sense of Buck and

\begin{equation*}
\mu(A(S,\{x_n\}))=\nu(S).
\end{equation*}
Thus by Theorem \ref{thm1} we have

\begin{equation*}
\nu(S) = \lim_{N \to \infty} \frac{1}{N} \left|\left\{n \le N ; x_{s_n} \in S \right\}\right|
\end{equation*}
for $S \in \D_\nu$ and $\{s_n\}$ a sequence of positive integers u.d.\ in $\Z$. Therefore the measure density can be represented in certain sense as "limit" density.

Consider now a uniformly continuous function $f : \XX \to [0,1]$ and a B.u.d.\ sequence $\{x_n\}_{n\in\mathbb{N}}$ in $\XX$. Then for every
real valued continuous function $g$ defined on $[0,1]$ we have
\begin{equation}
\lim_{N \to \infty}\frac{1}{N}\sum_{n=1}^N g(f(x_{k_n})) = \int f \circ g\ dP,
\end{equation}
where $\{k_n\}$ is an arbitrary sequence of positive integers u.d.\ in $\Z$. \\

More generally, let $\YY\neq \emptyset$ be a set and $\la $ be a pre-measure defined on the ring of all subsets of $\YY$. Assume further that $\la$ is a strong submeasure on $\D_\ell$ (a strong subadditive pre-measure, see\cite{Sipos}), i.e. 
\begin{equation*}
\la(A\cap B)+\la(A\cup B)\leq \la(A)+\la(B)\ ,
\end{equation*}
with $\la(\YY)=1, \la(\emptyset)=0$. Consider $\D_\ell$,
the system of all $C \subset \YY$ such that $\la(C) + \la(\YY \setminus C) =1$. Clearly, $\D_\ell$ is a set algebra and the restriction $\ell = \la|_{\D_\ell}$ is a finitely additive probability measure on $\D_\ell$.

Following the proofs from \cite{Pasteka3, Pasteka} we can derive the following result.

\begin{theorem}\label{thm7}
Let $g:\XX \to \YY$ be a bijective mapping. The following statements are equivalent
\begin{enumerate}
\item $g(S) \in \D_\ell \land \ell(g(S))=\nu(S)$, for all $S \in \D_\nu$; 
\item $\la(g(A)) \le \Delta(A)$, for all $A \in \D_0$;
\item $ì\la(g(B)) \le \na(B)$, for all $B\subset \XX$.
\end{enumerate}
\end{theorem}
Analogously we can extend the concept of B.u.d.\ for this case.
A sequence $\{x_n\}_{n\in\mathbb{N}}$ is said to be $\la$-{\it B.u.d.}\ if and only it for every $B \in \D_\ell$ we have
$A(B, \{x_n\}) \in \D_\mu$ and $\mu(A(B, \{x_n\}))=\ell(B)$.

In particular the following result holds.

\begin{theorem}\label{thm8}
If $g:\XX \to \YY$ is a bijective mapping fulfilling condition (1) of Theorem \ref{thm7} and $\{x_n \}$ is a $\la$-B.u.d.\ sequence of elements in $\YY$ then $g^{-1}(x_n)$ is a $\na$-B.u.d. sequence of elements in $\XX$.
\end{theorem}

For every bijection $g$, all $x \in \XX$ and $S \subset \XX$, we have
$g^{-1}(x) \in S \Leftrightarrow x \in g(S)$. Thus for each sequence $\{x_n\}_{n\in\mathbb{N}}$ we have $A(S,\{g^{-1}(x_n)\}) = A(g(S), \{x_n\})$. This yields

\begin{corollary}\label{corollary1}
If $g$ preserves measure density then for every $\na$-B.u.d.\ sequence $\{x_n\}_{n\in\mathbb{N}}$ the sequence $\{g^{-1}(x_n)\}_{n\in\mathbb{N}}$ is also $\na$-B.u.d..
\end{corollary}

\begin{remark}
Functions $g^{-1}$ such that the sequence $\{g^{-1}(x_n)\}_{n\in\mathbb{N}}$ is u.d.\ for every u.d.\ sequence $\{x_n\}_{n\in\mathbb{N}}$ are called uniform distribution preserving mappings (for short: u.d.p.\ mappings). They are of particular interest since they are maps generating u.d. sequences for every u.d.\ sequence $\{x_n\}_{n\in\mathbb{N}}$. In \cite{Tichy_Winkler} the authors establish general criteria on u.d.p.\ transformations on compact metric spaces and a full characterization of u.d.p. maps on $[0,1]$.
\end{remark}
Therefore, in view of Theorem \ref{thm7} we have the following

\begin{corollary}\label{corollary2}
Let $g$ be a bijection. Then the following statements are equivalent.
\begin{enumerate}
\item $g$ preserves the measure density;
\item $\na(g(A^{(n)}_i)) \le \Delta(A^{(n)}_i)$, for all $n \in \N$ and all $i \le k_n$; 
\item $\na(g(S)) \le \na(S)$, for all $S \subset \XX$.
\end{enumerate}
\end{corollary}

\begin{example}
Consider again $\XX = \G$ a locally compact abelian group.
The mapping $x \to x^{-1}$ defined on $G$ is a bijection fulfilling condition (2) of Corollary \ref{corollary2} by applying the decomposition
$A^{(n)}_j = a_jH_n$. Thus Corollary \ref{corollary1} implies that each sequence $\{x_n\}_{n\in\mathbb{N}}$ of elements of $\G$ is $\na$-B.u.d.\ sequence if and only if $\{x_n^{-1}\}$ is $\na$-B.u.d..\\
Analogously we can consider the mapping $x \to ax$, for a fixed $a \in G$. Then each sequence $\{x_n\}_{n\in\mathbb{N}}$ of elements of $\G$ is $\na$-B.u.d.\ sequence if and only if $\{ax_n\}$ is $\na$-B.u.d..
\end{example}
We now conclude this section with a theorem that can be considered as a generalization of the construction of Haar measure with the help of Kakutani's fixed point Theorem (see e.g. \cite{GD}).
\begin{theorem}
Let $g$ be a permutation defined on $\XX$ such that $g(S) \in \D_0$ for every $S \in \D_0$, where $\D_0$ is a countable $\sigma$-algebra of $\XX$. Then there exists a finite probability measure $\Delta$ such that for every $A \in \D_0$
\begin{equation*}
 \Delta(g(A)) = \Delta(A).
\end{equation*}
\end{theorem}
\begin{proof}
Denote by $\mathcal{B}$ the set of all bounded real valued and finitely additive set functions defined on $\D_0$. Then $\mathcal{B}$ is a linear space. Let $\mathcal{R}$ be the subset of $\mathcal{B}$ consisting of all finitely additive probability measures. It is easy to check that $\mathcal{R}$ is convex set.
Define a topology on $\mathcal{B}$ by
\begin{equation*}
\Lambda_n \rightarrow \Lambda \quad\Leftrightarrow \quad\forall S\in\D_0\quad \lim_{n\to\infty} \Lambda_n(S)=\Lambda(S)\ .
\end{equation*}
Consider a sequence $\{\Delta_n\}$ of elements in $\mathcal{R}$. Since $\D_0$ is a countable set, we can iteratively select a sequence of indices $\{n_k\}$ such that $\{\Delta_{n_k}(S)\}$ converges for every $S\in\D_0$. Thus $\mathcal{R}$ is sequentially compact with respect to this topology.\\
Let us define a linear mapping $\tilde{g}\colon\mathcal{B}\rightarrow \mathcal{B}$, with $\tilde{g}(\Lambda)(S) = \Lambda(g(S))$. Then $\tilde{g}(\mathcal{R})\subset \mathcal{R}$ and $\tilde{g}$ is continuous with respect to the topology under consideration.\\
Put
\begin{equation*}
\tilde{g}_n(\Delta)=\frac{1}{n}\sum_{k=1}^n \tilde{g}^n(\Delta)
\end{equation*}
for $\Delta\in \mathcal{R}$, $n = 1, 2, \dots$. Since $\mathcal{R}$ is sequentially compact, every countable centered system of closed sets has non empty intersection. Thus, by an application of Markov-Kakutani fixed point theorem, follows the assertion.
\end{proof}
The following example provides an explicit construction of a finite additive probability measure on the algebra $\D_0$.
\begin{example}
Let $\C$ be the set of all real-valued uniformly continuous functions defined on $\XX$. Since these functions are bounded we can define the norm
$$
||f||= \sup\{f(x);\ x \in \XX\},
$$
where $f \in \C$. It can be seen easily that $(\C, ||\cdot||)$ is a Banach space. 

Let $\C^\ast$ be the dual space to $\C$. Denote by ${\mathcal P}$ the set of all $\varphi \in \C^\ast$ that $\varphi(1)=1$ and
$\varphi(f) \ge 0$ for $f \ge 0$. Then $\Delta_{\varphi}(A)=\varphi(\chi_A), A \in \D_0$, $\varphi \in {\mathcal P}$, is a finitely additive probability measure on
$\D_0$ and $\varphi(f)=\int f d\Delta_{\varphi}$.

Assume that $g: \XX \to \XX$ is a permutation such that
$g(A) \in \D_0$ for $A \in \D_0$. We want to show that $g^{-1}\circ f\in \C$ for every $f\in \C$.\\ $f\in \C$ if and only for
every $\varepsilon >0$ there exists a step function $\sum_{i=1}^k c_i \chi_{A_i}, A_i \in \D_0$ that
$||f - \sum_{i=1}^k c_i \chi_{A_i}|| < \varepsilon $. Hence $||g^{-1}\circ f - \sum_{i=1}^k c_i \chi_{g(A_i)}|| < \varepsilon $, since
$g(A_i) \in \D_0$. Thus $g^{-1}\circ f$ is uniformly continuous.
\end{example}
\section{Buck uniform distribution mod 1}
In this section we study the connection between systems of measurable sets of positive integers and sets of real numbers in the unit interval. We will use the so-called radical-inverse function which is an important function in the theory of uniform distribution and in the study of low-discrepancy sequences (see for instance \cite{DT, KN}). Finally, the notion of upper Jordan measure mentioned in Example \ref{example3}, here denoted with $\la$, instead of $\na$, will be relevant.\\

Let $\XX=\N$ and $p$ a prime. Let us consider the arithmetic progression $r+(m)$. This leads to the system of decompositions

\begin{equation*}
\E_n = \{r+(p^n); r=0,\dots,p^n -1\}\ ,\quad n=1,2,\dots.
\end{equation*}
If $\Delta(r+(p^n))=\frac{1}{p^n}$, $n=1,2,\dots$, then the corresponding measure density $\na$ will be the covering density with respect to the system $\{p^n; n \in \N\}$, (see \cite{Pasteka4}).\\
Now, let us recall that every $n\in\N$ has a unique $p$-adic expansion , i.e. $n$ can be written as

\begin{equation*}
n=a_0(n) +a_1(n)p+\dots+a_s(n)p^s\ , \quad0\le a_i(n)<p\ ,\quad i=1,\dots,s.
\end{equation*}
The radical-inverse function $g_p:\N \rightarrow [0,1)$ is defined by

\begin{equation*}
g_p(n)=\frac{a_0(n)}{p} +\frac{a_1(n)}{p^2}+\dots+\frac{a_s(n)}{p^{s+1}}\ .
\end{equation*}
This function maps $\N$ to the set of $p$-adic rationals ${\mathbb J}_p = \{\frac{r}{p^s}; r=0,\dots,p^s-1\}$ in $[0, 1)$. Therefore the image of $\N$ under $g_p(n)$ is dense in $[0, 1)$. Since every number
 from ${\mathbb J}_p$ has finite $p$-adic expansion we obtain that the mapping $g_p : \N \to {\mathbb J}_p$ is a bijection.

The properties of $p$-adic expansions provide that

\begin{equation*}
g_p(r+(p^n))=\left[\frac{a}{p^n}, \frac{a+1}{p^n}\right)\cap {\mathbb J}_p
\end{equation*}
for $0\le r <p^n$ and $\frac{a}{p^n}=g_p(r)$. Let $\D_\ell$ be the system of all $S \subset {\mathbb J}_p$ such that
$\la(S)+\la({\mathbb J}_p \setminus S)=1$. Then $g_p$ and $\la$ satisfy condition (2) of Theorem \ref{thm7} which is equivalent to condition (1).\\

Let us remark that the sequence $(g_p(n))_{n\in\mathbb{N}}$, with $p$ not necessarily prime, is called the van der Corput sequence in base $p$ and it is a well-known example of u.d.\ sequence in $[0,1]$ (see \cite{DT, KN}). Moreover, the above construction has been considered and extended to more general systems of numeration by several researchers (see e.g. \cite{Carbone, HIT}). Recently, this method has been applied to obtain the so-called $LS$-sequences (see \cite{Carbone}). These sequences were first introduced in \cite{Carbone2} as sequences of points associated to the so-called $LS$-sequences of partitions of $[0, 1[$. The latter being obtained as a particular case of a splitting procedure introduced by Kakutani \cite{Kakutani} and generalized in \cite{Volcic}, for a particular choice of the parameters $L$ and $S$. Moreover, this construction has been generalized to the multidimensional case in \cite{Carbone_Volcic}.
Finally, let us note that when $L=p$ and $S=0$ the $LS$-sequence coincides with the van der Corput sequence in base $p$ (see \cite{AHZ}).\\

Now, let us consider the Cantor expansion. By this expansion every $x \in \N$ is uniquely given in the form

\begin{equation*}
x = b_1(x)1! + b_2(x)2!+\dots+b_s(x)s!, \ s \in \N, 0\le b_i(x) \le i,\ i=1,\dots,s.
\end{equation*}
Then we define a generalization of the radical-inverse function by
\begin{equation}\label{eq3}
g_v(x)=\frac{b_1(x)}{2!} + \frac{b_2(x)}{3!}+\dots+\frac{b_s(x)}{(s+1)!}. 
\end{equation}
Consider
\begin{equation*}
\E_n = \{r+(n!); r=0,\dots,n! -1\}\ ,\quad n=1,2,\dots,
\end{equation*}
as system of decompositions of $\N$, then $\na=\ma$-Buck measure density. Since every rational number in $[0,1)$ has finite Cantor expansion we observe that $g_v: \N \to {\mathbb J}$
is a bijective mapping. Clearly, for $n=1,2,\dots,$ and $r=0,\dots,n! -1$ we have
\begin{equation*}
g_v(r+(n!))= \left[\frac{b}{n!}, \frac{b+1}{n!}\right)\cap{\mathbb J}, \qquad g_v(r)= \frac{b}{n!}.
\end{equation*}
Again $\D_\ell$ is the set of all $S \subset {\mathbb J}$ such that $\la(S)+\la({\mathbb J} \setminus S)=1$, then $\la$ and $g_v$
fulfill the condition $(2)$ of Theorem \ref{thm7}.

Moreover, we observe that both $g_v$ and $g^{-1}_v$ satisfy $(1)$ of Theorem \ref{thm7}. Therefore  Theorem \ref{thm8} assures that a sequence $\{x_n\}_{n\in\mathbb{N}}$ of elements of ${\mathbb J}$ is a B.u.d.\ sequence if and only if
$\{g^{-1}_v(x_n) \}$ is a B.u.d.\ sequence in $\N$.\\

%%%%%%%%%%%%%%%%%%%%%%%%
Let us remark that the above example of Cantor expansion can be extended to general Cantor series, as shown in the following example.\\

Let $\left\{Q_n^{(k)}\right\}$ be increasing sequences of positive integers, $k=1, \dots, s$ such that $Q_n^{(k)}|Q_{n+1}^{(k)}, n=1,2,\dots$. Then every positive integer $a$ has a
unique representation in Cantor series of the form (see \cite{Cantor})
\begin{equation*}
a=\sum_{j=1}^{m_k} a^k_j Q_j^{(k)}, \qquad  0\le a^k_j < \frac{Q_{j+1}^{(k)}}{Q_j^{(k)}}.
\end{equation*}
Put
\begin{equation*}
\gamma_k(a)=\sum_{j=1}^{m_k} \frac{a^k_j}{Q_{j+1}^{(k)}}\ .
\end{equation*}
We can associate to $a$ a point in the $s$-dimensional unit interval
\begin{equation*}
 \gamma(a) = \left(\gamma_1(a), \dots, \gamma_s(a)\right).
\end{equation*}
If $J \subset [0,1]^s$ is the set of the form $J=J_1\times \dots\times J_s$ then
\begin{equation*}
A(\{\gamma(n)\}, J) = \bigcap_{k=1}^s A\left(\{\gamma_k(n)\}, J_k\}\right).
\end{equation*}
For $J_k = \left[\frac{i}{Q_n^{(k)}}, \frac{i+1}{Q_n^{(k)}}\right[$ we have
\begin{equation*}
A(\{\gamma_k(n)\}, J_k)= b + \left(Q_n^{(k)}\right)\ ,
\end{equation*}
with $b$ a suitable positive integer, and $b + \left(Q_n^{(k)}\right)$ the residue class of $b$ mod $Q_n^{(k)}$. Thus from the Chinese remainder theorem, if $Q_n^{(k_1)}, Q_n^{(k_2)}$, $k_1 \neq k_2$ are coprime, then $A(\{\gamma(n)\}, J)$ is Buck measurable and
\begin{equation}\label{chinese}
\mu(A(\{\gamma(n)\}, J) = |J_1|\dots|J_s|.
\end{equation}
Since the set of the points $\left(\frac{i_1}{Q_n^{(1)}}, ..., \frac{i_s}{Q_n^{(s)}}\right)$ is dense in $[0,1]^s$ we can conclude that \eqref{chinese} holds for arbitrary
intervals $J_k$, $k=1,\dots,s$.\\
%Examples of multidimensional u.d.\ sequences obtained starting by arbitrary numeration systems $\left\{Q_n^{(k)}\right\}$ are considered in \cite{HIT}. \\

%%%%%%%%%%%%%%%%%%%%%%%%%%%%%%%%%%%%%%%%
The above statements can be adapted to a more general setting. Let $f$ be a non-decreasing real-valued function on ${\mathbb J}$, with
$f(0)=0$, $f(1)=1$. For every $S \subset {\mathbb J}$, we can associate, in the usual way, the Jordan-Stieltjes upper measure $\la_f(S)$ defined as
$\la_f([a,b)\cap \mathbb J) = f(b)-f(a)$, with $a,b\in\Q$. By the generalized radical-inverse function $g$ defined in \eqref{eq3}, we can associate a finite additive measure on the system $\D_0$, where $\Delta_f(r+(n!))=\la_f(g_v(r+(n!)))$. On the other hand, if a finitely additive probability measure $\Delta$ on $\D_0$ is given, we can define a non-decreasing function $f(x)= \Delta(g^{-1}_v([0,x)\cap {\mathbb J})), x \in {\mathbb J} $, since every rational number can be expressed in the form $x=\frac{a}{n!}$ and so the preimage of $[0,x)\cap {\mathbb J}$ is a union of a finite number of sets of the form $r+(n!)$. Thus $\Delta = \Delta_f$.
If we denote by $\na_f$ the corresponding measure density, condition (1) of Theorem \ref{thm7} is fulfilled and Theorem \ref{thm8} yields the following

\begin{corollary}
A sequence of positive integers $\{y_n\}$ is $\nu_f$-B.u.d.\ if and only if $\{g_v(y_n)\}$ is $\ell_f$-B.u.d..
\end{corollary}
Moreover, we can prove the following result
\begin{lemma}\label{lemma3.2}
Let $f$ and $\Delta_f$ be as above. Then $f$ is uniformly continuous on ${\mathbb J}$ if and only if

\begin{equation}\label{eq4}
\lim_{n \to \infty} \Delta_f(r+(n!))=0
\end{equation}
uniformly for $r=0,1,\dots$.
\end{lemma}
\begin{proof}
Suppose that $[x_1, x_2) \in {\mathbb J}$. Clearly,

\begin{equation}\label{eq5}
f(x_2)-f(x_1) = \Delta_f(g^{-1}_v([x_1, x_2)\cap {\mathbb J})\ .
\end{equation}
Let $x_2 - x_1 <\frac{1}{(n+1)!}$ for some $n \in \N$. Then $[x_1, x_2) \subset [\frac{c}{n!}, \frac{c+1}{n!})$ for some $c\in\N$ with $0< c \le n!$. Thus $g^{-1}_v([x_1, x_2)\cap {\mathbb J} \subset r+(n!)$ for some $r\in\N$. Using \eqref{eq5}, this
yields $f(x_2)-f(x_1) \le \Delta_f(r+(n!))$. Hence $f$ is uniformly continuous on ${\mathbb J}$.\\
The other implication immediately follows  from

\begin{equation*}
\Delta_f(r+(n!))=\ell_f (g_v(r+(n!)) = \ell_f\Big(\Big[\frac{c'}{n!}, \frac{c'+1}{n!}\Big)\cap {\mathbb J}\Big)=f\left(\frac{c'+1}{n!}\right) - f\left(\frac{c'}{n!}\right)\ ,
\end{equation*}
with $r,n\in \N$, for a suitable non-negative integer $c'$.
\end{proof}

Since a uniformly continuous function on ${\mathbb J}$ can be extended to a continuous function on $[0,1]$, Lemma \ref{lemma3.2} has the following immediate consequence (see \cite[page 54]{Pasteka}).
\begin{corollary}
If $\{y_n\}$ is a $\nu$-B.u.d.\ sequence of positive integers fulfilling equation \eqref{eq4}, then the sequence $\{g_v(y_n)\}$ is Buck measurable and its Buck distribution function is the continuous extension of $f$ on $[0,1]$.
\end{corollary}

In the same way, one can prove the following result.
\begin{theorem}
Let $g : \XX \to [0,1]$  be an injective function such that $g(\XX)$ is dense in $[0,1]$ and assume that $g(A_r^{(n)}) = I_r^{(n)}\cap g(\XX)$, with $I_r^{(n)}$ right half-open intervals, and

\begin{equation*}
\lim_{n \to \infty} \ell(I_r^{(n)}) = 0
\end{equation*}
uniformly for $r \in \N$. \\
Denote $f(x) = \Delta(g^{-1}([0,x)\cap g(\XX))$ for every right endpoint $x$ of $I_r^{(n)}$ and for all $r,n\in\N$. Then $f$ is uniformly continuous on  $g(\XX)$ if and only

\begin{equation*}
\lim_{n \to \infty} \Delta(A_r^{(n)}) = 0
\end{equation*}
uniformly for $r \in \N$. \\
In this case for every $\nu$-B.u.d.\ sequence $\{y_n\}$ the sequence $\{g(y_n)\}$ is Buck measurable and its Buck distribution function is the continuous extension of $f$ to $[0,1]$.
\end{theorem}

It is well-known that a real-valued uniformly continuous function $f$ on a metric space $(\XX, \rho)$ can be extended to a continuous function on a compact space $\bar{\XX}$ (see \cite{Rudin}). In particular, one can define the concept of Riemann integrability by defining the Riemann upper and lower sums associated to the decompositions $\E_n, n=1,2,\dots$ and to the finitely additive measure $\Delta$. More precisely, we have the following definition.

\begin{definition}\label{def1}
Let $\{c_n\}$ be a sequence of positive integers such that for every $A \in \D_0$ there exists $n_0$ such that $A$ is a union of sets from $\E_{c_n}$, for $n \ge n_0$. Then a function $f$ is said to be Riemann integrable
if and only if there exists a real number $S$ such that for every system of finite sequences
$\{a^{(n)}_i; i=1,\dots,k_{c_n} \}$ with $a^{(n)}_i \in A^{(c_n)}_i$
we have
$$
\lim_{n \to \infty} \sum_{i=1}^{k_{c_n}} \Delta(A_i^{(c_n)})f(a^{(n)}_i) = S.
$$
In this case $S = \int f$.
\end{definition}
\begin{remark}
Let $B \subset \XX$ then $B$ is a $\na$-measurable set if and only if its indicator function $\chi_B$ is Riemann integrable and in this case
$$
\na(B) = \int \chi_B.
$$
\end{remark}
In this way, a sequence $\{x_n\}_{n\in\N}$ in $\XX$ is $\na$-B.u.d.\ if and only if for every Riemann integrable function
$f$ we have

\begin{equation*}
\lim_{N \to \infty} \sum_{n=1}^N f(x_{s_n}) = \int f
\end{equation*}
for every sequence of positive integers $\{s_n\}$ uniformly distributed in $\Z$.

Now, since $f$ is a real valued function uniformly continuous with respect to the metric $\rho$, we can extend it to a continuous function on $\bar{\XX}$ and
\begin{equation*}
\int f = \int f dP.
\end{equation*}
Thus Theorem \ref{thm7} can be extended to uniformly continuous functions $f$ with respect to $\rho$.

Under the assumption of continuity we can restate Theorem \ref{thm7}.

\begin{theorem}
 Let $g$ be a bijection such that $g^{-1}$ is uniformly continuous with respect to the metric $\rho$. Then $g$ preserves measure density if and only if there exists at least one $\na$-B.u.d.\ sequence $\{x_n\}_{n\in\mathbb{N}}$ such that $\{g^{-1}(x_n)\}$ is also
$\na$-B.u.d..
\end{theorem}

\section{Buck uniform distribution on a free semigroup}

Let  $\XX=\F$ be a free semigroup generated by a countable set of generators $\{p_1, p_2,\dots,p_n, \dots\}$. Let  $a \in \F$ and denote by $\U(a)$ the set of all divisors of $a$. We say that a set
$S \subset \F$ has a {\it divisor density} if and only if there exists

\begin{equation*}
\dv(S)=\lim_{n \to \infty}\frac{|S \cap \U(p_1^n...p_n^n)|}{(n+1)^n}\ .
\end{equation*}
We denote by $\D_{\dv}$ the family of all sets having a divisor density.\\

It is easy to show that $\dv$ is a finitely additive probability measure.\\ 
Let us consider some examples.\\
If $\mathcal{F}$ denotes the set of square free elements, then $|\mathcal{F} \cap \U(p_1^n...p_n^n)|=2^n$, hence $\mathcal{F} \in \D_{\dv}$ and $\dv(\mathcal{F})=0$.\\
Let $\F^s=\{a^s, a\in \N \}$, $s=2,3,\dots$. For $n$ sufficiently large we have
$|\F^s \cap \U(p_1^n...p_n^n)|=([n/s]+1)^n$. Thus $\F^s \in \D_{\dv}$ and $\dv(\F^s)=0$.\\
For $r\in\N$, denote by ${\mathbb O}_r$ the set of all elements of $\F$ of the form $p_{i_1}^{\beta_1}\dots p_{i_s}^{\beta_s}$, where $\beta_i \ge r$.

It can be easily seen that $|{\mathbb O}_r \cap \U(p_{1}^{n}\dots p_{n}^{n})|=(n-r+1)^n$. Thus ${\mathbb O}_r \in \D_{\dv}$ and
$\dv({\mathbb O}_r)=e^{-r}$, where $e$ is Euler's number.

\begin{proposition}
Let $S \in \D_{\dv}$ and $a \in \F$. Then $aS \in \D_{\dv}$ and $\dv(aS)=\dv(S)$.
\end{proposition}
\begin{proof}
It is sufficient to prove the assertion for $a=p$, one of the generators. Assume that $p$ occurs in $p_1,\dots,p_n$. If $d \in  S\cap \U(p_1^n\dots p_n^n)$ and $d=p_1^{j_1}\dots p_n^{j_n}$ then $pd$ does not divide
$p_1^n\dots p_n^n$ only in the case when the exponent of $p$ is $n$. Thus
$|pS\cap \U(p_1^n\dots p_n^n)|= S\cap \U(p_1^n\dots p_n^n)-(n+1)^{(n-1)}$.
\end{proof}

\begin{corollary}
For $b \in \F$ we have $b\F \in  \D_{\dv}$ and $\dv(b\F)=1$.\\
 For every $S \subset \F$, $b \in \F$ we have that $(b\F)\cap S \in  \D_{\dv}$ implies $S \in \D_{\dv}$ and $\dv(S)=\dv((b\F)\cap S)$.
\end{corollary}

\begin{proposition}\label{prop2}
Let $H \subset \F$ be the semigroup generated by the generators\linebreak
$\{p_1^{\alpha_1},\dots, p_k^{\alpha_k}, p_{k+1},p_{k+2},\dots.\}$ (for given positive integers $\alpha_i$). Then $H \in \D_{\dv}$ and
$\dv(H)=\frac{1}{\alpha_1\dots\alpha_k}$ .
\end{proposition}

\begin{proof}
Let $n> \alpha_i, i=1,\dots,k$, $n>k$. Then $H\cap \U(p_1^n...p_n^n)$ contains the elements
$p_1^{s_1\alpha_1},\dots, p_k^{s_k\alpha_k}p_{k+1}^{j_{k+1}}\dots p_n^{j_n}$, where $s_i\alpha_i \le n,\ i=1,\dots,k$
and $j_i \le n,\ i=k+1,\dots,n$.
Thus
$$
|H\cap \U(p_1^n...p_n^n)|=([n/ \alpha_1]+1)\dots([n/ \alpha_k]+1)(n+1)^{(n-k)}.
$$
Thus
$$
\lim_{n\to \infty} \frac{|H\cap \U(p_1^n\dots p_n^n)|}{(n+1)^n}=\frac{1}{\alpha_1\dots\alpha_k}.
$$
\end{proof}

A mapping $x: \F \to [0,1]$ is called $\dv$-{\it uniformly distributed} if and only if for every
subinterval $I \subset [0,1)$ the set $x^{-1}(I)$ belongs to $\D_{\dv}$ and $\dv(x^{-1}(I))= \ell(I)$.

Now we can formulate the following generalization of Weyl's criterion.

\begin{proposition}\label{prop3}
Let $P_n=p_1^n\dots p_n^n$, $n=1,2,\dots$.
A mapping $x:\F \to [0,1)$ is $\dv$- uniformly distributed if and only if one of the following conditions holds:
\begin{equation}\label{eq6}
\lim_{n \to \infty} (n+1)^{-n}n\sum_{d|P_n} f(d) = \int^1_0 f(x)dx
\end{equation}
for every Riemann integrable function $f$ (or for continuous $f$)  on $[0,1]$;
\begin{equation}\label{eq7}
\lim_{n \to \infty} (n+1)^{-n}n\sum_{d|P_n} e^{2\pi i h x(d)} = 0 
\end{equation}
 for integers $h\neq 0$.
\end{proposition}

%%%%%%%%%%%%%%

Let $\F_n$ be the semigroup generated by $\{p_1^n, p_2^n, \dots p_n^n, p_{n+1},p_{n+2}, \dots\}$, $n=1,2,\dots$.\\
Proposition \ref{prop2} implies that $a\F_k \in \D_{\dv}$ and $\dv(a\F_k)=\frac{1}{k^k}$.\\
Let $\D_0$ be the system of all subsets of $\F$ of the form $a_1\F_n \cup \dots\cup a_k\F_n$, where
$a_j \in \{p_1^{j_1}p_2^{j_2} \dots p_n^{j_n}, 0\le j_i < n \}$. Put $\Delta(A)= \dv(A)$ for $A \in \D_0$.
For $S \subset \F$ the value
$$
\na_e(S)= \inf \{\Delta(A); S \subset A, A \in \D_0\}
$$
is called {\it divisor measure density} of the set $S$. We denote by $\D_{\nu_e}$ the system of measurable sets.
It can be easily deduced that $\D_{\nu_e} \subset \D_{\dv}$ and $\dv(S)=\nu_e(S)$ for $S \in \D_{\nu_e}$.\\

Let us consider again $\mathcal{F}$, the set of all square free elements of $\F$.  This set has non-empty intersection only with the sets $aF_n$, $n=1,2,\dots$ and $a=p_1^{i_1} p_2^{i_2} \dots p_n^{i_n}$, where $i_j=0,1$.  Hence, $\na_e(\mathcal{F})\le \frac{2^n}{n^n}$, and for $n \to \infty$ we get
$\na_e(\mathcal{F})=0$. This yields $\na_e(\F \setminus \mathcal{F})=1$. So $\mathcal{F}$ is measurable and $\nu_e(\mathcal{F})=0$. Let us remark that the set of square free integers is not Buck measurable, its Buck measure density is $\frac{6}{\pi^2}$, and its complement has Buck measure density $1$. However, the set of square free numbers has asymptotic density $\frac{6}{\pi^2}$.

\begin{example}
For the set $\mathcal{F}$  we have $\int \chi_{\mathcal{F}} d\nu_e=0$. However, this function is discontinuous at each point
$a \in \mathcal{F}$.
\end{example}

Denote by $[S:\F_{n}]$ the number of sets $a_j\F_n$ such that $S \cap a_j\F_n \neq \emptyset$. Theorem \ref{thm1} implies 
\begin{proposition}\label{prop4}
 Let $\{b_k\}$ be a sequence of positive integers such that for every $d \in \N$ there exists $k_0$ such that for $k > k_0$ $d|b_k$. Then for $S \subset \F$
$$
\na_e(S) = \lim_{k \to \infty}\frac{[S:\F_{b_k}]}{b_k^{b_k}}.
$$
\end{proposition}
In the sequel $\{b_k\}$ will be a sequence as in Proposition \ref{prop4}.
Denote $\F^s=\{a^s; s \in \F\}$ for  $s\in\N$. Suppose that $s|b_k$. Then the intersection $\F^s$ with
$p_1^{j_1}p_2^{j_2} \dots p_{b_k}^{j_{b_k}}\F_{b_k}$ is non-empty only if $s|j_i$ for all $i=1,\dots,b_k$. Therefore
\begin{equation*}
\na(\F^s) = \lim_{k \to \infty}\frac{(b_k/s)^{b_k}}{b_k^{b_k}}=\lim_{k \to \infty}\frac{1}{s^{b_k}}
\end{equation*}
and so for $s>1$ we get $\na_e(\F^s)=0$.

Let $\F'$ be the semigroup generated by $p_{j_1}, p_{j_2}, \dots,p_{j_n},\dots$. Then $p_1^{j_1}p_2^{j_2} \dots p_{b_k}^{j_{b_k}}\linebreak\F_{b_k} \cap \F'\neq \emptyset$ if and only if the exponents $j_i\neq 0$ are exactly the generators occurring in the sequence of generators of $\F'$. Denote by $R(k)$ the number of $r_n$ belonging to
$1,2,\dots, b_k$. Then
$$
\na_e(\F') = \lim_{k \to \infty}\Big(\frac{{b_k}^{R(k)}}{b_k^{b_k}}\Big).
$$
And so if $R(k) < b_k$ then $\na_e(\F')=0$.

Denote by ${\mathbb O}_r$ for given $r$ - positive integer the set of all elements of $\F$ in the form $p_{i_1}^{\beta_1}...p_{i_s}^{\beta_s}$ where
$\beta_i \ge r$. Then ${\mathbb O}_r \cap p_{1}^{j_1}...p_{n}^{j_n}\F_n \neq \emptyset$ only when $j_i>r$ or $j_i=0$. Thus for $n>r$ we have
$[{\mathbb O}_r:F_n]=(n-r)^n$. By application of Proposition I we get $\na_e({\mathbb O}_r)={\bold e}^{-r}$, ${\bold e}$ is Euler's number. From the other side
$(\F \setminus {\mathbb O}_r) \cap p_{1}^{j_1}...p_{n}^{j_n}\F_n \neq 0$ for all cases. Thus $\na_e(\F \setminus {\mathbb O}_r)=1$,
${\mathbb O}_r$ is not measurable.

Choose a sequence of positive integers
$\{n_k\}$ with $n_k^{n_k} | n_{k+1}^{n_{k+1}}$, in the same way as in \cite[page 42]{Pasteka}. It can be proved that $\nu_e$ has the Darboux property on $\D_{\nu_e}$, (see also \cite{PS}). Thus
$\{\nu_e(A); A \in \D_{\nu_e} \}=[0,1]$.

The following result follows immediately from Definition \ref{def1}.

\begin{proposition}\label{prop5}
A real valued function $f$ defined on $\F$  is Riemann integrable if and only if there exists a real number $\alpha$ such that for every system of finite sequences
$\{a^{(n)}_i; i=1,\dots,b_n^{b_n} \}$ with $a^{(n)}_i \in a_i\F_{b_n}$

$$
\lim_{n \to \infty} b_n^{-b_n}\sum_{i=1}^{b_n^{b_n}} f(a^{(n)}_i) = \alpha
$$
holds.
In this case $\alpha = \int f d\nu_e$.
\end{proposition}
A mapping $x: \F \to [0,1]$ is called {\it uniformly divisor measurable} if and only if for every
subinterval $I \subset [0,1)$ the set $x^{-1}(I)$ belongs to $\D_{\nu_e}$ and $\nu_e(x^{-1}(I))= \ell(I)$.

The mentioned mapping is a real valued net defined on $\F$ (considered as a partially ordered set).

From the definition it follows immediately: If $\{y_n\}$ is $\nu_e$-B.u.d.\ sequence in $\F$ and $x: \F \to [0,1]$ is a uniformly divisor measurable map, then $\{x(y_n)\}$ is a B.u.d.\ sequence in $[0,1]$.
\begin{example}
Let $\F = \N$ and $\{b_k\}$ fulfilling additionally the ralations $b_k^{b_k}|b_{k+1}^{b_{k+1}}$ for $k\in\N$. Consider the following partition of the unit interval
$$
[0,1) = \bigcup_{j=1}^{b_k^{b_{k}}} I_j^{(k)}, \ I_j^{(k)}= \left[\frac{j-1}{b_k^{b_{k}}}, \frac{j}{b_k^{b_{k}}}\right)
$$
for $k=1,2,\dots$.

Suppose that the sequence $\{y_n\}$ of elements $[0,1)$ satisfies :

\begin{equation*}
y_n \in I_j^{(k)} \Longleftrightarrow n \in a_jF_{b_k}.
\end{equation*}

Then $y^{-1}(I_j^{(k)}) =A\left(\{y_n\}, I_j^{(k)}\right) = a_jF_{b_k}$. Thus $A(\{y_n\}, I_j^{(k)}) \in \D_{\nu_e}$ and $\nu_e(A(\{y_n\}, I_j^{(k)})) =\ell(I_j^{(k)})$.
The set $\left\{\frac{j}{b_k^{b_{k}}}; j=1,\dots, b_k^{b_{k}}, k=1,2,\dots\right\}$ is dense in $[0,1)$ and so $\{y_n\}$ is uniformly divisor measurable.
For the details see \cite{H}.
\end{example}

For every mapping $x: \F \to [0,1]$ we have $\chi_{x^{-1}(I)}(a) = \chi_I(x(a))$ for every $a$  and $I\subset [0,1)$.
We can state the generalization of Weyl's criterion in this case.
\begin{proposition}
A mapping $x:\F \to [0,1)$ is uniformly divisor measurable if and only for every system of finite sequences
$\{a^{(n)}_i; i=1,\dots,b_n^{b_n} \}$ with $a^{(n)}_i \in a_i\F_{b_n}$ 

\begin{equation*}
\lim_{n \to \infty} b_n^{-b_n}\sum_{i=1}^{b_n^{b_n}} f(x(a^{(n)}_i)) = \int^1_0 f(x)dx
\end{equation*}
holds for every Riemann integrable function $f$ on $[0,1]$.
\end{proposition}
\begin{proposition}
Let $H \subset \F$ be the semigroup generated by the generators
$\{p_1^{\alpha_1},\dots, p_k^{\alpha_k}, p_{k+1},p_{k+2},\dots\}$, $\alpha_i>0$. Then $H \in \D_{\nu_e}$ and
$\nu_e(H)=\frac{1}{\alpha_1...\alpha_k}$ .
\end{proposition}
\begin{proof}
Consider $n>k$ and $n$ divisible by all $\alpha_i$. The $\F_n \subset H$ and
$$
H = \bigcup p_1^{j_1\alpha_1}...p_k^{j_k\alpha_k}p_{k+1}^{j_{k+1}}...p_n^{j_{n}}\F_n
$$
where $j_1\alpha_1<n,\dots,j_k\alpha_k <n , j_{k+1}<n,\dots,j_{n}<n$. Thus $H \in \D_{\nu_e}$ and
$$
\nu_e(H)=\frac{n}{\alpha_1}\dots\frac{n}{\alpha_k}n^{(n-k)} \cdot n^{-n}=\frac{1}{\alpha_1...\alpha_k}.
$$
\end{proof}
\begin{proposition}
Let $S \in \D_{\nu_e}$ and $a \in F$. Then $aS \in \D_{\nu_e}$ and $\nu_e(aS)=\nu_e(S)$.
\end{proposition}
\begin{proof}
It is sufficient to prove this statement for $a=p_i$, one of the generators. Let $n>i$.
If the exponent of $p_i$ in $b$ does not exceed $n-1$ then $p_ib\F_n \in \D_{\nu_e}$. If $b =p_i^{n-1}b'$ and
$p_i$ not occurring in $b'$, then $p_ib\F_n \subset b'F_n$. Thus $\na_e(p_ib\F_n) \le n^{-n}$.\\
Suppose that for $\varepsilon > 0$ there exists $n>i$ such that
 $$
 \bigcup_{j=1}^k a_jF_n \subset S \subset \bigcup_{j=1}^s a_jF_n, \ \frac{s-k}{n^n} < \varepsilon.
 $$
 Then
 $$
 \bigcup_{j=1}^k p_ia_jF_n \subset p_iS \subset \bigcup_{j=1}^s p_ib_jF_n.
 $$
 The sequences $a_j$, $j=1,\dots, k$, $b_j, j=1,\dots, s$ contain at most $n^{n-1}$ elements divisible by $p_i^{n-1}$. Thus
 $$
 \na_e\Big(\bigcup_{j=1}^s p_ib_jF_n \setminus \bigcup_{j=1}^k p_ia_jF_n\Big) \le
 \frac{s+n^{n-1}}{n^n}-\frac{k-n^{n-1}}{n^n} <\varepsilon +\frac{2}{n}.
$$
\end{proof}

 Let $g$ be a bijection on the set of generators. We can extend this mapping to an automorphism of $\F$. Proposition \ref{prop4} and Proposition \ref{prop5}
 provide that $g$ fulfills condition (2) of Corollary \ref{corollary2}. We have proved that $g$ preserves divisors measure density.
On the other hand each automorphism $\F$ is uniquely determined by its values on the set of generators. Thus each automorphism on $\F$ preserves divisor measure density or $\nu_e$ does not depend on the order of generators.
\begin{corollary}
For $b \in \F$ it holds $b\F \in  \D_{\nu_e}$ and $\nu_e(b\F)=1$. If $S \subset \F$ and $b \in \F$ then $S \in  \D_{\nu_e}$ if and only if $(b\F)\cap S \in  \D_{\nu_e}$.
 In this case $\nu_e(S)=\nu_e(b\F\cap S)$.
\end{corollary}
\section*{Acknowledgments}
The first and third author would like to acknowledge the support of the Austrian Science Fund (FWF): F5510. They are also grateful to Prof. O. Strauch from the Slovak Academy of Science and Prof. V. Bal\'{a}\v{z} from Comenius University in Bratislava for the fruitful discussions they had during their visit in Bratislava in November 2014.
\bibliography{ipt}
\bibliographystyle{abbrv}
\end{document}